\documentclass[a4paper]{amsart}
        \usepackage{latexsym}
        \usepackage{amssymb}
        \usepackage{amsmath}
        \usepackage{amsfonts}
        \usepackage{amsthm}
        \usepackage[hypertexnames=false]{hyperref}
        \ifpdf
          \usepackage[all,knot,poly,2cell]{xy}
        \else
          \usepackage[all,knot,poly,2cell,dvips]{xy}
        \fi
             \UseAllTwocells
        \usepackage{xspace}

        \usepackage{eucal}

        \theoremstyle{plain}
        \newtheorem{theorem}{Theorem}[section]
        \newtheorem{corollary}[theorem]{Corollary}
        \newtheorem{lemma}[theorem]{Lemma}
        \newtheorem{proposition}[theorem]{Proposition}

        \theoremstyle{definition}
        \newtheorem{definition}[theorem]{Definition}
        \newtheorem{example}[theorem]{Example}

        \theoremstyle{remark}
        \newtheorem{remark}[theorem]{Remark}
        
        \newcommand{\itemref}[1]{\eqref{#1}}

        \newcommand{\Z}{\mathbb{Z}}

        \newcommand{\Orb}{\mathcal{O}}   
       \DeclareMathOperator{\spec}{Spec} 
        \newcommand{\red}[1]{{#1}_{\mathrm{red}}} 

           \newcommand{\Aff}{\mathbb{A}}

        \newcommand{\DCAT}{\mathsf{D}}

        \newcommand{\QCOH}{\mathsf{QCoh}}

\newcommand{\HS}[2]{\underline{\mathrm{HS}}_{{#1}/{#2}}}

\renewcommand{\subset}{\subseteq}
\numberwithin{equation}{section}

\newcommand{\qcsubscript}{\mathrm{qc}} 

\newcommand{\DQCOH}[1][]{\DCAT_{\qcsubscript{#1}}} 

\newcommand{\spref}[1]{\href{http://stacks.math.columbia.edu/tag/#1}{#1}}

\makeatletter
\newcommand{\labitem}[2]{%
\def\@itemlabel{(\textbf{#1})}
\item
\def\@currentlabel{\textbf{#1}}\label{#2}}
\makeatother
\title{Addendum: \'Etale d\'evissage, descent and pushouts of stacks}
\date{Nov 28, 2017}
\author[J. Hall]{Jack Hall}
\address{Department of Mathematics\\University of Arizona\\Tucson, AZ 85721\\USA}
\email{jackhall@math.arizona.edu}
\author[D. Rydh]{David Rydh}
\address{KTH Royal Institute of Technology\\Department of Mathematics\\SE-100 44 Stockholm\\Sweden}
\email{dary@math.kth.se}
\thanks{%
The first author was supported by the Australian Research Council DE150101799.
The second author was supported by the Swedish Research Council 2011-5599 and 2015-05554.}

\newcommand{\cms}{\mathrm{cs}}
\newcommand{\STACK}{\mathbf{Stack}}
\usepackage{mathrsfs}
\newcommand{\stX}{\mathscr{X}}%
\newcommand{\sepdiag}{\mathrm{sep}_\Delta}%
\usepackage{stmaryrd}
\newcommand{\thickslash}{\mathbin{\!\!\pmb{\fatslash}}}
\theoremstyle{plain}
\newtheorem{innercustomthm}{Theorem}
\newenvironment{customthm}[1]
  {\renewcommand\theinnercustomthm{#1}\innercustomthm}
  {\endinnercustomthm}
\begin{document}
\begin{abstract}
  Using Nisnevich coverings and a Hilbert stack of stacky points, we prove \'etale d\'evissage results for non-representable \'etale and quasi-finite flat coverings. We give applications to absolute noetherian approximation of algebraic stacks and compact generation of derived categories.

\end{abstract}
\maketitle
\section{Introduction}
In \cite[Thm.~D \& 6.1]{MR2774654}, d\'evissage results were proved for representable quasi-finite flat and \'etale morphisms. We will show how these results may be extended to the non-representable situation using Nisnevich coverings and a Hilbert stack of stacky points.

We apply these results to weaken the separation hypotheses from the approximation results for algebraic stacks that appeared in~\cite{rydh-2009} and the compact generation result for derived categories of quasi-coherent sheaves on Deligne--Mumford stacks that appeared in~\cite[Thm.~A]{perfect_complexes_stacks}.

The results of this article have already been used in \cite{2017arXiv170502295H}. We also expect further applications arising from the work of \cite{AHR_lunafield,etale_local_stacks} on the local structure of stacks near points with linearly reductive stabilizers, where non-representable \'etale coverings naturally arise (see Remark \ref{R:application_LUNA}). 

Before stating our main result, we require some notation. Fix an algebraic stack~$S$. If 
$P_1$, $\ldots$, $P_r$ is a list of properties of morphisms of algebraic stacks over~$S$, let 
$\STACK_{P_1 , \ldots ,P_r /S}$ denote the full $2$-subcategory of the $2$-category of algebraic 
stacks over~$S$ whose objects are those $(x\colon X \to S)$ such that $x$ has properties 
$P_1$, $\ldots$, $P_r$. The following abbreviations will be used: \'et (\'etale), qff 
(quasi-finite flat), sep (separated), fp (finitely presented), rep (representable), and $\sepdiag$ (separated diagonal). Throughout, we let $\mathbf{E} \subseteq \STACK_{/S}$ be one of the following $2$-subcategories:
\[
\xymatrix{%
  \STACK_{\mathrm{repr,sep,fp,\acute{e}t}/S}\ar@{}[r]|-*[@]{\subseteq}\ar@{}[d]|-*[@]{\subseteq} &
  \STACK_{\mathrm{sep,fp,\acute{e}t}/S}\ar@{}[r]|-*[@]{\subseteq}\ar@{}[d]|-*[@]{\subseteq} &
  \STACK_{\mathrm{\sepdiag,fp,\acute{e}t}/S}\ar@{}[d]|-*[@]{\subseteq} \\
  \STACK_{\mathrm{repr,sep,fp,qff}/S}\ar@{}[r]|-*[@]{\subseteq} &
  \STACK_{\mathrm{sep,fp,qff}/S}\ar@{}[r]|-*[@]{\subseteq} &
  \STACK_{\mathrm{\sepdiag,fp,qff}/S}.
}%
\]
Our improvement of \cite[Thm.~D \& 6.1]{MR2774654} is the following theorem.
\begin{customthm}{D$'$}[\'Etale or quasi-finite flat d\'evissage]\label{T:qff_devissage}
  Let $S$ be a quasi-compact and quasi-separated algebraic stack and
  let $\mathbf{E}$ be as above. Let $(T' \xrightarrow{t} T) \in \mathbf{E}$ be surjective (resp.~surjective and representable) and let
  $\mathbf{D} \subseteq \mathbf{E}$ be a full $2$-subcategory satisfying the following three conditions: 
  \begin{enumerate}
  \item[(D1)] if $(X' \to X) \in \mathbf{E}$ is \'etale and $X \in
    \mathbf{D}$, then $X' \in \mathbf{D}$;
  \item[(D2)] if $(X' \to X) \in \mathbf{E}$ is proper (resp.~finite) and surjective
    and $X' \in \mathbf{D}$, then $X \in \mathbf{D}$; and
  \item[(D3)] if $(U \xrightarrow{i} X)$, $(X'\xrightarrow{f} X) \in
    \mathbf{E}$, where $i$ is an open immersion and $f$ is \'etale and an isomorphism
    over $X\setminus U$, then $X \in \mathbf{D}$ whenever $U$,
    $X'\in\mathbf{D}$.
  \end{enumerate}
If $T'\in \mathbf{D}$, then $T \in \mathbf{D}$. 
\end{customthm}
\begin{proof}
  Combine Theorem \ref{T:nis_pres_qff} with Lemma \ref{L:nis_dev}.
\end{proof}

Note that if $(X'\to X)\in \mathbf{E}$ is \'etale, then there is a canonical
factorization $X'\to X''\to X$ in $\mathbf{E}$ where the first morphism is
an \'etale gerbe and the second morphism is \'etale. If in addition $(X' \to X)$ is proper, then $X' \to X''$ is a proper \'etale gerbe and $X'' \to X$ is finite \'etale.

Note that if $T' \to T$ is representable, then it has separated
diagonal. In particular, the advantage of Theorem
\ref{T:qff_devissage} over \cite[Thm.~D]{MR2774654} is the removal of
the assumption of representability from $T' \to T$.

The ``Induction principle'' \cite[Tag \spref{08GL}]{stacks-project} for algebraic spaces is closely related to the d\'evissage results of Theorem \ref{T:qff_devissage}. When working with derived categories or K-theory, where locality results are often quite subtle, it is often advantageous to have the strongest possible criteria at your disposal (e.g., \cite{hall_balmer_cms}). For stacks with quasi-finite diagonal, we also obtain the following Induction principle. 
\begin{customthm}{E}[Induction principle for stacks with quasi-finite diagonal]\label{T:induction_qf_diag}
  Let $S$ be a quasi-compact and quasi-separated algebraic stack. Choose $\mathbf{E} \subseteq \STACK_{/S}$ as follows:
  \begin{itemize}
  \item if $S$ has quasi-finite diagonal, take $\mathbf{E} =
   \STACK_{\mathrm{\sepdiag,fp,qff}/S}$;
  \item if $S$ has quasi-finite and separated diagonal, take $\mathbf{E} =
    \STACK_{\mathrm{repr,sep,fp,qff}/S}$;
   \item if $S$ is Deligne--Mumford, take $\mathbf{E} =
    \STACK_{\mathrm{\sepdiag,fp,\acute{e}t}/S}$; and
  \item if $S$ is Deligne--Mumford with separated diagonal, take $\mathbf{E} =
\STACK_{\mathrm{repr,sep,fp,\acute{e}t}/S}$.
  \end{itemize}
  Let $\mathbf{D} \subseteq \mathbf{E}$ be a full $2$-subcategory satisfying
  the following properties:
  \begin{enumerate}
  \item[(I1)] if $(X' \to X) \in \mathbf{E}$ is an open immersion and $X \in
    \mathbf{D}$, then $X' \in \mathbf{D}$;
  \item[(I2)] if $(X' \to X) \in \mathbf{E}$ is finite and surjective,
    where $X'$ is an affine scheme, then $X \in \mathbf{D}$; and
  \item[(I3)] if $(U \xrightarrow{i} X)$, $(X'\xrightarrow{f} X) \in
    \mathbf{E}$, where $i$ is an open immersion and $f$ is \'etale and an isomorphism
    over $X\setminus U$, then $X \in \mathbf{D}$ whenever $U$,
    $X'\in\mathbf{D}$. 
  \end{enumerate}
  Then $\mathbf{D} = \mathbf{E}$. In particular, $S \in \mathbf{D}$.
\end{customthm}
\begin{proof}
  Combine Lemma \ref{L:nis_dev} with Theorem \ref{T:qfdiag_pres}.
\end{proof}
We wish to point out that Theorem \ref{T:induction_qf_diag} relies on the existence of coarse spaces for stacks with finite inertia (i.e., the Keel--Mori Theorem \cite{MR1432041,MR3084720}). Theorem \ref{T:induction_qf_diag}, in the case of a separated diagonal, was proved in \cite[App.~B]{hall_balmer_cms}.
\begin{remark}
  Extending Theorem \ref{T:qff_devissage} to covers with non-separated
  diagonals is possible. The most natural and useful formulation,
  however, requires $2$-stacks and the corresponding notion of
  $2$-Nisnevich coverings. This is analogous to the situation of
  representable but non-separated coverings, where non-representable
  Nisnevich coverings naturally appear. See Remark \ref{R:nonsep} for more details.
\end{remark}
\subsection*{Conventions}
We make no a priori separation assumptions on our algebraic stacks, just as in \cite{stacks-project}. 
\section{Residual gerbes as intersections}
Let $X$ be a quasi-separated algebraic stack (e.g., $X$ noetherian). By
\cite[Thm.~B.2]{MR2774654}, every point of $X$ is algebraic. That is, if $x\in
|X|$, then there is a quasi-affine monomorphism $\mathcal{G}_x \to X$ with
image $x$ such that $\mathcal{G}_x$ is an fppf gerbe, the \emph{residual
  gerbe}. Using the recent approximation result~\cite{rydh-2014}, which
depends on the original \'etale d\'evissage~\cite{MR2774654}, we obtain

\begin{lemma}\label{L:residual-gerbe-as-inv-limit}
Let $X$ be a quasi-separated algebraic stack and let $x\in |X|$ be a point.
The residual gerbe $\mathcal{G}_x$ is the limit of an
inverse system of immersions $j_\lambda\colon U_\lambda\hookrightarrow X$ of
finite presentation with affine bonding maps.
\end{lemma}
\begin{proof}
There is a locally closed integral substack $Z\hookrightarrow X$ such that $Z$
is a gerbe over an affine scheme $\underline{Z}$ and $x$ is the generic point
of $Z$~\cite[Thm.~B.2]{MR2774654}. Let $U\subseteq X$ be a quasi-compact open
neighborhood of $Z$ such that $Z\hookrightarrow U$ is a closed immersion.
Consider the inverse system $\{W_\lambda\hookrightarrow U\}_{\lambda\in
  \Lambda}$ of all finitely presented affine immersions
$W_\lambda\hookrightarrow U$ such that $x\in |W_\lambda|$.
We claim that the inverse limit, i.e., the
intersection, is $\mathcal{G}_x$.

Indeed, let $\pi\colon Z\to \underline{Z}$ denote the structure map of the
gerbe. Then $\pi(x)$ is the intersection of its affine open neighborhoods
$\underline{Z}_\alpha\subseteq \underline{Z}$. Thus
$\mathcal{G}_x=\pi^{-1}(\spec \kappa(\pi(x)))$ is the intersection of its
relatively affine open neighborhoods $Z_\alpha=\pi^{-1}(\underline{Z}_\alpha)$,
i.e., the open immersions $Z_\alpha\hookrightarrow Z$ are affine.  Moreover,
for a fixed $\alpha$, we may pick an open quasi-compact substack
$U_\alpha\subseteq U$ such that $Z_\alpha=Z\cap U_\alpha$. Since
$Z_\alpha\hookrightarrow U_\alpha$ is a closed immersion, we may write
$Z_\alpha\hookrightarrow U_\alpha$ as the intersection of closed immersions
$Z_{\alpha\beta}\hookrightarrow U_\alpha$ of finite
presentation~\cite{rydh-2014}. For sufficiently large $\beta$, the immersion
$Z_{\alpha\beta}\hookrightarrow U_\alpha\hookrightarrow U$ is affine, since the
limit $Z_\alpha\hookrightarrow U_\alpha\hookrightarrow U$ is
affine~\cite[Thm.~C]{rydh-2009}. Thus $Z_{\alpha\beta}=W_\lambda$ for some
$\lambda=\lambda(\alpha,\beta)$ for every $\alpha$ and every sufficiently large
$\beta$. It follows that
\[
\mathcal{G}_x\hookrightarrow\bigcap_{\lambda\in \Lambda}
W_\lambda\hookrightarrow \bigcap_{\alpha} Z_\alpha=\mathcal{G}_x
\]
and the result follows.
\end{proof}

\section{Nisnevich d\'evissage}\label{App:nis}
In this section, we consider Nisnevich coverings for quasi-separated
algebraic stacks. For schemes, this goes back to the work of
\cite{MR1045853} with the most famous applications due to
\cite{MR1813224}. In the setting of equivariant schemes this was
considered in \cite[\S2]{MR3431674}. It was also considered for
Deligne--Mumford stacks in \cite[\S\S 7-8]{MR2922391}.
The restriction to quasi-separated algebraic stacks is so that we can
give an intuitive definition in terms of residual gerbes.

\begin{definition}
A morphism of quasi-separated algebraic stacks
$p\colon W \to X$ is a \emph{Nisnevich covering} if it is \'etale and for every
$x\in |X|$, there exists an $w \in |W|$ such that $p(w) = x$ and the induced
map of residual gerbes $\mathcal{G}_{w}\to \mathcal{G}_{x}$ is an isomorphism.
\end{definition}

Nisnevich coverings are stable under composition and base change.

\begin{example}\label{ex:nisnevich_scheme}
  Let $X$ be a quasi-compact and quasi-separated scheme. Then there exists an affine scheme $W$ and a Nisnevich covering $p\colon W \to X$. Indeed, taking $W=\coprod_{i=1}^n U_i$, where the $\{U_i\}$ form a finite affine open covering of $X$ gives the claim. More generally, this holds for quasi-compact and quasi-separated algebraic spaces~\cite[Prop.~5.7.6]{MR0308104}. 
\end{example}

Let $p \colon W \to X$ be a morphism of algebraic stacks. Recall that when
$p$ is not representable, then a section of $p$ need not be a monomorphism.
A \emph{monomorphic splitting sequence} for $p$ is a sequence of
quasi-compact open immersions
\[
\emptyset=X_0 \subseteq X_1 \subseteq \cdots \subseteq X_{r} = X
\]
such that $p$ restricted to $X_{i}\setminus X_{i-1}$, when given the induced reduced structure, admits a monomorphic section
for each $i=1$, $\dots$, $r$. In this situation, we say that $p$ has a
monomorphic splitting sequence of length $r$.

We have the following characterization of Nisnevich coverings, which is well-known for noetherian schemes \cite[Lem.~3.1.5]{MR1813224}.
\begin{proposition}\label{P:nisn_splitting}
  Let $X$ be a quasi-compact and quasi-separated algebraic stack
  and let $p\colon W \to X$ be a quasi-separated \'etale morphism.
  Then $p$ is a Nisnevich covering if and
  only if there exists a monomorphic splitting sequence for $p$.
\end{proposition}
\begin{proof}
Let $x\in |X|$ be a point. Then there exists an immersion $Z_x\hookrightarrow
X$ of finite presentation, such that $x\in |Z_x|$, and a monomorphic section of
$p|_{Z_x}$. Indeed, there is a monomorphic section of $p|_{\mathcal{G}_x}$
which extends to a monomorphic section of $p|_{Z_x}$ by
Lemma~\ref{L:residual-gerbe-as-inv-limit} and~\cite[Prop.~B.2 (i) and B.3
  (ii)]{rydh-2009}.

The $Z_x$ are constructible and we can thus cover $X$ by a finite
number of the $Z_x$'s. We can thus filter $X$ by a sequence of quasi-compact
open substacks $X_i$ such that $X_i\setminus X_{i-1}$ is contained in some
$Z_x$. That is, we have obtained a monomorphic splitting sequence.
\end{proof}

The following lemma outlines the key benefits of the Nisnevich topology: it is generated by particularly simple coverings (cf.~\cite[Prop.~1.4]{MR1813224}).
\begin{lemma}[Nisnevich d\'evissage]\label{L:nis_dev}
  Let $S$ be a quasi-compact and quasi-separated algebraic stack and
  let $\mathbf{E}\subseteq \STACK_{\mathrm{fp,\acute{e}t}/S}$ be a full
  $2$-subcategory containing all open immersions and closed under fiber products
  (e.g., one of the categories listed in the introduction). Let
  $\mathbf{D} \subseteq \mathbf{E}$ be a full $2$-subcategory such that
  \begin{enumerate}
  \item[(N1)] if $(X' \to X) \in \mathbf{E}$ is an open immersion and
    $X \in \mathbf{D}$, then $X' \in \mathbf{D}$; and
  \item[(N2)] if $(U \xrightarrow{i} X)$, $(X'\xrightarrow{f} X) \in
    \mathbf{E}$, where $i$ is an open immersion and $f$ is an isomorphism
    over $X\setminus U$, then $X \in \mathbf{D}$ whenever $U$,
    $X'\in\mathbf{D}$.
  \end{enumerate}
  If $p \colon W \to X$ is a Nisnevich covering in $\mathbf{E}$ and $W \in
  \mathbf{D}$, then $X \in \mathbf{D}$.
\end{lemma}
\begin{proof}
  By Proposition \ref{P:nisn_splitting}, there is a sequence of quasi-compact
  open immersions:
  \[
  \emptyset=X_0 \subseteq X_1 \subseteq \cdots \subseteq X_{r}=X,
  \]
  such that $f$ restricted to $X_{i}\setminus X_{i-1}$, when given the
  induced reduced structure, admits a monomorphic section for $i=1$,
  $\dots$, $r$. We will prove the result by induction on $r\geq 0$. If $r=0$, then the result is trivial.

  If $r>0$, let $U=X_{r-1}$; then $U$ admits a splitting
  sequence of length $r-1$. By the inductive hypothesis and (N1), we
  may thus assume that $U \in \mathbf{D}$. If $Z=\red{(X\setminus U)}$, then
  the restriction of $p$ to $Z$ admits a section $s$, which is a
  quasi-compact open immersion. It follows that $X'=p^{-1}(U)\cup s(Z)
  = W\setminus (p^{-1}(Z) \setminus s(Z))$ is a quasi-compact open
  subset of $W$. Let $f\colon X' \to X$ be the induced morphism;
  then $X'\in \mathbf{D}$ and $f$ is an isomorphism over $X\setminus
  U$. By (N2), the result follows.
\end{proof}
\section{Presentations of algebraic stacks with finite stabilizers}
The following theorem removes the separated diagonal assumption from \cite[Thm.~B.5]{hall_balmer_cms}.  It will be crucial for the proofs of Theorems \ref{T:induction_qf_diag} and \ref{T:hilbqfb}. 
\begin{theorem}\label{T:qfdiag_pres}
  Let $X$ be a quasi-compact and quasi-separated algebraic stack with quasi-finite
  diagonal. Then there exist morphisms of algebraic stacks
  \[
  V \xrightarrow{v} W \xrightarrow{p} X
  \]
  such that
  \begin{itemize}
  \item $V$ is an affine scheme;
  \item $v$ is finite, faithfully flat and of finite presentation; and
  \item $p$ is a Nisnevich covering of finite presentation with separated diagonal.
  \end{itemize}
 In addition,
  \begin{enumerate}
  \item\label{TI:qfdiag_pres:sep-diag} if $X$ has separated diagonal, then it can be arranged that
    $p$ is representable and separated; and
  \item if $X$ is Deligne--Mumford, then it can be arranged that $v$ is \'etale.
  \end{enumerate}
\end{theorem}
\begin{proof}
  The proof is similar to \cite[Prop.~6.11]{MR3084720}, \cite[Thms.~6.3 \& 7.2]{MR2774654} and \cite[Thm.~B.5]{hall_balmer_cms}. 

  By \cite[Thm.~7.1]{MR2774654}, there is an affine scheme $U$ and a
  representable, quasi-finite, faithfully flat and finitely presented morphism $u\colon U
  \to X$. The Hilbert stack $\HS{U}{X}=\amalg_{d\geq 0}\mathscr{H}^d_{U/X}\to X$
  parametrizing quasi-finite representable morphisms to $U$
  is algebraic and has quasi-affine---in particular, separated---diagonal
  \cite[Thm.~4.4]{MR2821738}.
  Let $p\colon W=\HS{U}{X}^{\mathrm{\acute{e}t}} \to X$ be the open substack of
  the Hilbert stack that parameterizes representable \'etale morphisms to $U$.
  Since $u$ is flat, it is readily seen that $p\colon W \to X$ is \'etale.

  We now prove that $p$ is a Nisnevich covering. Let $x\in |X|$ be a point with
  residual gerbe $\mathcal{G}_x$. The restriction $u_x\colon U_x\to
  \mathcal{G}_x$ is finite and flat. Thus, the identity $U_x\to U_x$
  corresponds to a section $\mathcal{G}_x\to W$. It is readily seen that this
  is a monomorphic section (e.g., by considering the open substack $H\subseteq
  W$ below).

  After replacing $W$ by a quasi-compact open subset containing the sections
  of a monomorphic splitting sequence (Proposition~\ref{P:nisn_splitting}),
  we obtain a finitely presented Nisnevich covering $p\colon W \to X$. Let
  $v\colon V \to W$ be the universal family, which is finite (even \'etale if
  $u$ is \'etale), flat and of finite presentation. Then there is a
  $2$-commutative diagram
  \[
    \xymatrix{V \ar[r]^q \ar[d]_v & U \ar[d]^u \\ W \ar[r]^p & X,}
  \]
  where $p$ and $q$ are \'etale.
  After shrinking $W$, we may assume that $v$ is surjective.
  Although $p$ and $q$ need neither be
  representable nor separated, we saw that $p$, and hence $q$, have separated
  diagonals. It follows that $V$ has separated diagonal, and hence so has
  $W$~\cite[Lem.~A.4]{MR2774654}. We may replace $X$ by $W$
  and assume that $X$ has separated diagonal.

  When $X$ has separated diagonal, the presentation $u$ is separated.
  Consider the substack $H=\underline{\mathrm{Hilb}}_{U/X}^{\mathrm{open}}
  \subseteq W$ parameterizing open and closed immersions into $U$ over
  $X$. In general $H$ is not algebraic but since $u$ is separated it
  is an open substack of $W$ and $H\to S$ is representable and
  separated~\cite[Thm.~4.1]{MR2821738}. We may thus replace
  $W$ with a quasi-compact open subset of $H$ containing the sections. Then we
  obtain a commutative diagram as above where $p$ and $q$ are \'etale,
  representable and separated. By Zariski's Main Theorem
  \cite[Thm.~A.2]{MR1771927}, $q$ is quasi-affine. By
  \cite[Thm.~5.3]{MR3084720}, $W$ has a coarse space $\pi\colon W \to W_{\cms}$
  such that $W_{\cms}$ is a quasi-affine scheme and $\pi\circ v$ is affine
  (and integral). By
  Example \ref{ex:nisnevich_scheme}, we may further reduce to the
  situation where $W_{\cms}$ is an affine scheme. Then $V$ is affine and
  the result follows.
\end{proof}
\begin{remark}
A special case of~\itemref{TI:qfdiag_pres:sep-diag} is when $X$ has finite
inertia. Then one can give an alternative proof of Theorem~\ref{T:qfdiag_pres}
using that $X$ admits a coarse space $X\to X_{\cms}$ and that
Nisnevich-locally on $X_{\cms}$, we can find a finite flat presentation of
$X$. Indeed, one immediately reduces to the case where $X_{\cms}$ is local
henselian and then a quasi-finite flat presentation $U\to X$ splits as
$U=V\amalg V'$ where $V\to X$ is finite and surjective.
\end{remark}
\section{Hilbert stack of stacky points}
Let $f\colon X \to S$ be a morphism of algebraic stacks. Let
$\HS{X}{S}$ be the \emph{Hilbert stack} of $f$. The Hilbert stack of
$f$ parameterizes quasi-finite and representable morphisms to $X$ that
are proper over the base. In \cite{hallj_dary_g_hilb_quot,MR3148551},
it was proved that $\HS{X}{S}$ was algebraic when $f$ has quasi-finite
and separated diagonal. The proof of this relies on the results of
\cite{MR3148551}, whose methods are quite involved and may not be so
familiar to the reader.

In this article, we will only need a small piece of $\HS{X}{S}$: the
open substack $\HS{X}{S}^{\mathrm{qfb}}$ consisting of those families
that are quasi-finite (though not necessarily representable) over the
base. We will call this the \emph{Hilbert stack of stacky
  points}. Using Nisnevich coverings, we will be able to deduce the
algebraicity of the Hilbert stack of stacky points from the well-known
algebraicity result in the case where $f$ is separated, which is much
easier (e.g, \cite{MR2233719}, \cite[Thm.~9.1]{MR3589351} and
\cite[Thm.~A(i)]{hallj_dary_g_hilb_quot}).
\begin{theorem}\label{T:hilbqfb}
  If $f\colon X \to S$ is a morphism of algebraic stacks with
  quasi-compact and separated diagonal, then
  $\HS{X}{S}^{\mathrm{qfb}}$ is an algebraic stack with quasi-affine
  diagonal over $S$. If $f$ is locally of finite presentation (resp.\
  is separated), then $\HS{X}{S}^{\mathrm{qfb}}$ is locally of finite
  presentation (resp.\ has affine diagonal).
\end{theorem}
To prove Theorem \ref{T:hilbqfb} we first prove a result on Weil
restrictions.

\begin{proposition}\label{P:weilr}
  Let $Z\to S$ be a quasi-finite, proper and flat morphism of finite
  presentation between quasi-separated algebraic stacks.
  If $U\to Z$ is a quasi-separated morphism
  with quasi-finite diagonal, then the Weil restriction
  $\mathbf{R}_{Z/S}(U)\to S$ is a quasi-separated algebraic stack.
  Moreover, if $U\to Z$ is
\begin{enumerate}
\item\label{PI:weilr:Nis} a Nisnevich covering; or
\item\label{PI:weilr:etale} \'etale; or
\item\label{PI:weilr:rep} representable; or
\item\label{PI:weilr:rep+sep} representable and separated; or
  \item\label{PI:weilr:qc} quasi-compact,
\end{enumerate}
then so too is $\mathbf{R}_{Z/S}(U)\to S$. If $U \to Z$ has separated diagonal, then $\mathbf{R}_{Z/S}(U) \to S$ has quasi-affine diagonal. 
\end{proposition}
  If $U \to Z$ has separated diagonal, it can be deduced that
  $\mathbf{R}_{Z/S}(U)$ is algebraic with quasi-affine diagonal using
  \cite[Thm.~2.3(vi)]{hallj_dary_g_hilb_quot}. This relies on
  \cite{MR3148551}, however. We will avoid the reliance on \cite{MR3148551}
  and the separated
  diagonal assumption when $Z\to S$ is quasi-finite using
  a simple bootstrapping process and Theorem \ref{T:qfdiag_pres}.

\begin{proof}[Proof of Proposition~\ref{P:weilr}]
  A standard argument shows that
  properties~\itemref{PI:weilr:etale}, \itemref{PI:weilr:rep}, and \itemref{PI:weilr:rep+sep} are preserved by taking Weil restrictions whenever the Weil restrictions in question exist, cf.\ \cite[Rem.~2.5]{hallj_dary_g_hilb_quot}. To prove~\itemref{PI:weilr:Nis} when $\mathbf{R}_{Z/S}(U) \to S$ is already known to be a quasi-separated algebraic stack, 
  we may replace $S$ with a residual gerbe $\mathcal{G}_s$ for some point
  $s\in |S|$. Then
  $|Z|$ is finite and discrete. Thus, if $U\to Z$ is a
  Nisnevich covering, then $U\to Z$ has a monomorphic section. It follows that
  there is a monomorphic section $S\to \mathbf{R}_{Z/S}(U)$. 
  
  We make the following well-known observation: if
  $u\colon U_1 \to U_2$ is a morphism of algebraic stacks over $Z$,
  then the base change of
  $\mathbf{R}_{Z/S}(u) \colon \mathbf{R}_{Z/S}(U_1) \to
  \mathbf{R}_{Z/S}(U_2)$ along a morphism
  $T \to \mathbf{R}_{Z/S}(U_2)$, corresponding to a $Z$-morphism
  $Z\times_S T\to U_2$, is isomorphic to
  $\mathbf{R}_{Z\times_S T/T}((Z\times_S T)\times_{U_2}
  U_1)$. It follows that if $P$ is a property of morphisms
  of algebraic stacks that is smooth-local on the target, then
  $\mathbf{R}_{Z/S}(u)$ is $P$ if
  $\mathbf{R}_{Z/S}(U) \to S$ is $P$ for all
  affine $S$ and all $U \to Z$ satisfying $P$.

  We next address the algebraicity. If $U \to Z$ is separated
  (resp.~separated and representable), then
  $\mathbf{R}_{Z/S}(U) \to S$ is well-known to be algebraic with
  affine diagonal (resp.~representable and separated), see
  \cite[Thm.~2.3(v)]{hallj_dary_g_hilb_quot}.

  The algebraicity is smooth local on $S$, so we may assume that $S$
  is an affine scheme. Every section of $U\to Z$ factors through a
  quasi-compact open subset and Weil-restrictions of open substacks are open
  substacks, hence we may assume that $U$ is quasi-compact.
  Theorem~\ref{T:qfdiag_pres} implies that there
  is a Nisnevich covering $p\colon W \to U$ such that $W$ has finite
  diagonal and $W \to U$ has separated diagonal. By the case already
  considered, $\mathbf{R}_{Z/S}(W) \to S$ is algebraic with affine
  diagonal. Consider the induced morphism
  $\mathbf{R}_{Z/S}(p) \colon \mathbf{R}_{Z/S}(W) \to
  \mathbf{R}_{Z/S}(U)$.

  If $U \to Z$ has separated diagonal, then
  Theorem~\ref{T:qfdiag_pres} even says that we can choose the
  Nisnevich covering $p\colon W \to U$ to be separated and
  representable. The separated case already considered and
  \itemref{PI:weilr:Nis}--\itemref{PI:weilr:rep+sep} now establishes
  that $\mathbf{R}_{Z/S}(p)$ is a representable and separated
  Nisnevich covering. Hence, $\mathbf{R}_{Z/S}(U) \to S$ is
  algebraic. To see that it has quasi-affine diagonal, we note that
  $\mathbf{R}_{Z/S}(U) \times_S \mathbf{R}_{Z/S}(U) \cong
  \mathbf{R}_{Z/S}(U\times_Z U)$. In particular,
  $\Delta_{\mathbf{R}_{Z/S}(U)} \simeq
  \mathbf{R}_{Z/S}(\Delta_{U/Z})$. Since $\Delta_{U/Z}$ is
  quasi-affine, $\mathbf{R}_{Z/S}(\Delta_{U/Z})$ is quasi-affine
  \cite[Thm~2.3(iii)]{hallj_dary_g_hilb_quot}.

  If $U \to Z$ does not have separated diagonal, then
  $p\colon W \to U$ still has separated diagonal. Hence, by the cases
  already considered, $\mathbf{R}_{Z/S}(p)$ is algebraic and a Nisnevich \'etale
  covering. It follows that $\mathbf{R}_{Z/S}(U)$ is algebraic, but we
  still need to prove that it is quasi-separated. Repeating the
  argument above on separation conditions for
  $\mathbf{R}_{Z/S}(U) \to S$, the quasi-separatedness follows
  from~\itemref{PI:weilr:qc}.

It remains to show~\itemref{PI:weilr:qc}: the Weil restriction $R:=\mathbf{R}_{Z/S}(U)\to S$ is
quasi-compact if $U\to Z$ is quasi-compact. This claim is smooth local on $S$
so we may assume that $S$ is affine. Pick a quasi-finite flat presentation
$Z'\to Z$ and let $Z''=Z'\times_Z Z'$ and $Z'''=Z'\times_Z Z'\times_Z Z'$.  To
show that $R$ is quasi-compact, we may replace $S$ with a stratification. We
may thus assume that $Z'\to S$ is finite.
Then $R':=\mathbf{R}_{Z'/S}(U\times_Z Z')\to S$,
$R'':=\mathbf{R}_{Z''/S}(U\times_Z Z'')\to S$
and $R''':=\mathbf{R}_{Z'''/S}(U\times_Z Z''')\to S$
are quasi-compact and quasi-separated algebraic stacks
\cite[Prop.~3.8 (xiii) \& (xix)]{MR2821738}. If we define $P$ (descent data without
the descent condition) by the cartesian square
\[
\xymatrix{
R'\ar[d]_{(\pi_1^*,\pi_2^*)} & P\ar[l]\ar[d] \\
R''\times_S R'' & R''\ar[l]_-{\Delta}\ar@{}[ul]|\square
}
\]
then there is a cartesian square
\[
\xymatrix{
P\ar[d]_{\tau} & R\ar[l]\ar[d] \\
I_{R'''} & R'''\ar[l]_{e}\ar@{}[ul]|\square
}
\]
by fppf descent~\cite[Rmk.~4.4]{MR2357471}. It follows that
$R$ is quasi-compact.
\end{proof}

We can now prove Theorem \ref{T:hilbqfb}.
\begin{proof}[Proof of Theorem \ref{T:hilbqfb}]
  We may assume that $S$ is an affine scheme. If $X^{\mathrm{qf}} \subseteq X$ denotes the open substack where $X$ has a quasi-finite diagonal, then it is clear that $\HS{X^{\mathrm{qf}}}{S}^{\mathrm{qfb}} =\HS{X}{S}^{\mathrm{qfb}}$; thus we may assume that $X$ has quasi-finite and separated diagonal. Further standard reductions permit us to assume that $X$ is also quasi-compact. By Theorem \ref{T:qfdiag_pres}, there is a finitely presented, representable, and separated Nisnevich covering $p\colon W \to X$ such that $W$ admits a finite flat and finitely presented covering by an affine scheme $V$.  If $X$ is separated, we instead let $W=X$. In either case, $W$ has finite diagonal. By \cite[Thm.~A(i)]{hallj_dary_g_hilb_quot}, $\HS{W}{S}^{\mathrm{qfb}}$ is an algebraic stack with affine diagonal.

Let $T$ be an affine scheme and let $(Z \to X\times_S T) \in \HS{X}{S}^{\mathrm{qfb}}(T)$. It is well-known that the following diagram is $2$-cartesian:
\[
\xymatrix{\mathbf{R}_{Z/T}((W\times_S T)\times_{X\times_S T} Z) \ar[r] \ar[d] & T  \ar[d] \\ \HS{W}{S}^{\mathrm{qfb}} \ar[r] & \HS{X}{S}^{\mathrm{qfb}}, } 
\]
and we conclude that $\HS{W}{S}^{\mathrm{qfb}} \to \HS{X}{S}^{\mathrm{qfb}}$ is
a finitely presented, representable, and separated Nisnevich covering (Proposition~\ref{P:weilr}). The theorem follows.
\end{proof}
\begin{example}
  Theorem \ref{T:hilbqfb} is false if $X \to S$ has non-separated
  diagonal. This is similar to the main result of \cite{MR2369042}
  (cf.~\cite{MR3148551}). For an explicit example, consider
  $S=\Aff^1_k$, where $k$ is a field, and let $G = (\Z/2\Z)_S$. Let
  $H \subseteq G$ be the \'etale subgroup scheme which is the
  complement of the non-trivial element lying over the origin in
  $S$. The quotient $G/H$ is non-separated (it is just the line with
  the doubled origin). Let $X=B_S(G/H)$. Let
  $S_n = \spec (k[x]/x^{n+1})$ and $\hat{S} = \spec k[[x]]$. The
  natural map $(B_SG)\times_S S_n \to X\times_S {S_n}$ is
  representable (even an isomorphism), but there is no extension of
  this to a representable morphism $Y \to X\times_S {\hat{S}}$, where
  $Y \to \hat{S}$ is proper and flat.
\end{example}
\begin{remark}\label{R:nonsep}
  If $X \to S$ is non-separated, then the natural object to consider
  is the $2$-stack parameterizing not necessarily representable
  morphisms $Z \to X$ that are quasi-finite and flat over the
  base. This $2$-stack ends up being algebraic because the proof of
  Theorem \ref{T:hilbqfb} holds verbatim. If $X\to S$ is flat and
  we restrict to the
  $2$-substack parameterizing those $Z \to X$ that are also \'etale,
  then this is an \'etale $2$-stack. In particular, it is an \'etale
  $2$-gerbe over a $1$-stack. Unfortunately, this $1$-stack does not
  carry a universal
  family, which makes applying the result difficult. In particular, to
  prove d\'evissage results for morphisms with non-separated
  diagonals, it appears necessary to enter the world of higher
  stacks, cf.\ Remark~\ref{R:nonsep2}.
\end{remark}
\section{Non-representable presentations}
The following theorem combines and extends \cite[Prop.~6.11]{MR3084720} and
\cite[Thm.~6.3]{MR2774654}. It makes crucial use of Theorem \ref{T:hilbqfb}. 
\begin{theorem}\label{T:nis_pres_qff}
  Let $X$ be a quasi-compact and quasi-separated algebraic stack and
  let $u \colon U \to X$ be a quasi-finite and
  faithfully flat morphism of finite presentation with separated diagonal. Then there exists a
  commutative diagram of algebraic stacks
  \[
  \xymatrix{V \ar[r]^q \ar[d]_v & U \ar[d]^u \\ W \ar[r]^p & X}
  \]
  such that
  \begin{itemize}
  \item $v$ is quasi-finite, proper and faithfully flat of finite
    presentation;
  \item $p$ is a Nisnevich \'etale covering of finite presentation with separated diagonal; and
  \item $q$ is an \'etale morphism of finite presentation with separated diagonal.
  \end{itemize}
  In addition,
  \begin{enumerate}
  \item\label{T:nis_pres_qff:rep} if $u$ is representable, then it
    can be arranged that $v$ is representable;
  \item\label{T:nis_pres_qff:sep} if $u$ is separated, then it can be
    arranged that $p$ and $q$ are separated and representable; and
  \item\label{T:nis_pres_qff:et} if $u$ is \'etale, then it can be
    arranged that $v$ is \'etale.
  \end{enumerate}
\end{theorem}
\begin{proof}
  Argue exactly as in the proof of the first part of Theorem \ref{T:qfdiag_pres}.
  As before we take $W=\HS{U}{X}^{\mathrm{\acute{e}t}}$, the open substack of
  the Hilbert stack $\HS{U}{X}$ parameterizing \'etale morphisms to $U$.
  Since $U\to X$ is quasi-finite, $\HS{U}{X}=\HS{U}{X}^{\mathrm{qfb}}$ is
  algebraic with quasi-affine diagonal (Theorem~\ref{T:hilbqfb}). As before,
  it follows that $W\to X$ is \'etale with quasi-affine, hence separated,
  diagonal. If $u$ is separated, we replace $W$
  with the open substack $\underline{\mathrm{Hilb}}_{U/X}^{\mathrm{open}}$
  which is separated and representable over $X$.
\end{proof}
\begin{remark}\label{R:nonsep2}
If $u$ does not have separated diagonal in Theorem~\ref{T:nis_pres_qff},
then using the Hilbert $2$-stack of Remark~\ref{R:nonsep}, we would
arrive at the conclusion of the Theorem except that
$p$ and $q$ need not have separated diagonals and are merely $2$-representable,
though $v$ is still $1$-representable.
Here $n$-representable means represented by algebraic $n$-stacks.
In particular, $V$ and $W$ are algebraic $2$-stacks.
\end{remark}
\section{Applications}
In this section, we use non-representable \'etale d\'evissage to relax
some separatedness conditions in the approximation results
of~\cite{rydh-2009} and the compact generation results
of~\cite{perfect_complexes_stacks}.
\begin{lemma}\label{L:approximation-of-gerbes}
Let $S$ be a quasi-compact and quasi-separated algebraic stack.  Let $X$ be a
quasi-compact and quasi-separated algebraic stack over $S$ and let $\pi\colon
\stX\to X$ be a proper fppf gerbe. Suppose
$\stX=\varprojlim_{\lambda\in\Lambda} \stX_\lambda$ where $\stX_\lambda$ are algebraic
stacks of finite presentation over $S$ and $g_\lambda\colon \stX\to \stX_\lambda$ are affine
morphisms. Then for all sufficiently large $\lambda$, there is a commutative diagram
\[
\xymatrix{\stX\ar[d]_{\pi}\ar[r]\ar@/^1pc/[rr]^{g_\lambda} & \stX^\circ_\lambda\ar[d]^{\pi_\lambda}\ar[r]_{i_\lambda} & \stX_\lambda\\
X\ar[r] & X^\circ_\lambda\ar@{}[ul]|\square}
\]
where $i_\lambda$ is a finitely presented closed immersion,
$\pi_\lambda$ is a proper fppf gerbe and the square is cartesian.
In particular, $X\to X^\circ_\lambda$ is affine and $X^\circ_\lambda\to S$ is of
finite presentation.
\end{lemma}
\begin{proof}
The map $\pi$ gives an exact sequence of group objects over $\stX$
\[
0\to I_{\stX/X}\to I_{\stX/S}\to \pi^*I_{X/S}.
\]
That $\pi$ is an fppf gerbe of finite presentation implies that $I_{\stX/X}$ is
flat and of finite presentation. Conversely, given a flat subgroup $G\subset
I_{\stX/S}$ of finite presentation, there exists a \emph{rigidification}: an
algebraic stack $\stX\thickslash G$ over $S$ together with an fppf gerbe
$\stX\to \stX\thickslash G$ of finite presentation such that the relative
inertia is $G$~\cite[Thm.~A.1]{MR2427954}.

Let $G=I_{\stX/X}$ and fix an index $\alpha\in \Lambda$. The inertia stack of
$I_{\stX_\alpha/S}$ does not pull-back to $I_{\stX/S}$ but the canonical map
$I_{\stX/S}\to I_{\stX_\alpha/S}\times_{\stX_\alpha} \stX$ is a closed
subgroup stack. Since $G\to \stX$ and $I_{\stX_\alpha/S}\to \stX_\alpha$ are of finite
presentation, there is, by standard approximation methods~\cite[Props.~B.2,
  B.3]{rydh-2009}, an index $\lambda\geq \alpha$ and a subgroup
$G_\lambda\hookrightarrow I_{\stX_\alpha/S}\times_{\stX_\alpha} \stX_\lambda$ of
finite presentation that pulls back to $G\hookrightarrow
I_{\stX_\alpha/S}\times_{\stX_\alpha} \stX$. After
increasing $\lambda$, we may assume that $G_\lambda\to \stX_\lambda$ is
flat and proper~\cite[Prop.~B.3]{rydh-2009}.

We now address the problem that $G_\lambda$ need not be a subgroup of
$I_{\stX_\lambda/S}$. Let $H_\lambda=G_\lambda\cap I_{\stX_\lambda/S}$ as subgroups of
$I_{\stX_\alpha/S}\times_{\stX_\alpha}
\stX_\lambda$. Then $H_\lambda\to G_\lambda$ is a finitely presented closed subgroup and $H_\lambda\times_{\stX_\lambda} \stX\to G_\lambda\times_{\stX_\lambda} \stX$ is an
isomorphism. It follows that the Weil restriction
$\stX^\circ_\lambda := \mathbf{R}_{G_\lambda/\stX_\lambda}(H_\lambda)$ is a finitely presented
closed substack of $\stX_\lambda$ and that $g_\lambda\colon \stX\to \stX_\lambda$ factors
uniquely through $\stX^\circ_\lambda$. Also note that after restricting to $\stX^\circ_\lambda$,
the closed subgroup $H_\lambda\to G_\lambda$ becomes an isomorphism. We thus have
the subgroup $G^\circ_\lambda:=G_\lambda|_{\stX^\circ_\lambda}\to I_{\stX^\circ_\lambda/S}$ which is proper
and flat over $\stX^\circ_\lambda$.

Let $X^\circ_\lambda=\stX^\circ_\lambda\thickslash G^\circ_\lambda$. It remains to prove that we have a
cartesian diagram. Since $\stX\to X$ is initial among maps $\stX\to Y$ such
that $G\hookrightarrow I_{\stX/S}$ factors through $I_{\stX/Y}\hookrightarrow
I_{\stX/S}$, we have a map $X\to X^\circ_\lambda$. This induces a map between gerbes
$\stX\to \stX^\circ_\lambda\times_{X^\circ_\lambda} X$ over $X$. This is a stabilizer-preserving
morphism, i.e., $I_{\stX/X}=G\to I_{\stX^\circ_\lambda/X^\circ_\lambda} \times_{\stX^\circ_\lambda} \stX=G^\circ_\lambda\times_{\stX^\circ_\lambda} \stX$ is an
isomorphism. But a
stabilizer-preserving morphism between gerbes is an isomorphism.
\end{proof}

We can now remove most of the representability assumption
in~\cite[Lemma.~7.9]{rydh-2009}.
\begin{proposition}\label{P:non-rep-et-descent-of-approx}
Let $S$ be a pseudo-noetherian stack and let $X\to S$ be a morphism of algebraic
stacks. Let $W\to X$ be an \'etale surjective morphism of finite presentation with separated diagonal (e.g., representable). If $W\to S$ can be approximated, then so can
$X\to S$.
\end{proposition}
\begin{proof}
We will apply \'etale d\'evissage (Theorem~\ref{T:qff_devissage}).  Let
$\mathbf{D} \subseteq \mathbf{E}=\STACK_{\mathrm{\sepdiag,fp,\acute{e}t}/S}$ be the full
subcategory of morphisms $Y\to X$ such that $Y\to S$ is of strict approximation
type or, equivalently, has an approximation~\cite[Prop.~4.8]{rydh-2009}. Then
(D1) is satisfied by definition; (D2) for finite morphisms is~\cite[Prop.~2.12
  (ii)]{rydh-2009} and (D3) is~\cite[Lem.~7.8]{rydh-2009}. It remains to prove
(D2) for proper non-representable morphisms. Thus, let $Y'\to Y$ be a proper
\'etale surjective morphism in $\mathbf{E}$. There is a canonical factorization
$Y'\to Y''\to Y$ where the first morphism is an \'etale gerbe and the second
is finite \'etale. It is thus enough to prove (D2) when $Y'\to Y$ is a proper
\'etale gerbe.

By assumption $Y'\to S$ has an approximation and can thus be written as
$Y'=\varprojlim_\lambda {Y'_\lambda}$ where $Y'_\lambda\to S$ are of finite
presentation and $Y'\to Y'_\lambda$ is affine for every $\lambda$. By
Lemma~\ref{L:approximation-of-gerbes} we have a cartesian diagram
\[
\xymatrix{Y'\ar[d]\ar[r] & Y'^\circ_\lambda\ar[d]\\
Y\ar[r] & Y^\circ_\lambda.\ar@{}[ul]|\square}
\]
of algebraic stacks over $S$ where $Y\to Y^\circ_\lambda$ is affine and $Y^\circ_\lambda\to
S$ is of finite presentation. Thus, $Y\to S$ has an approximation.
\end{proof}

In~\cite{rydh-2009} it is shown that quasi-compact algebraic stacks with
quasi-finite and \emph{locally separated} diagonal can be approximated and are
pseudo-noetherian. We can now remove the locally separatedness assumption.

\begin{corollary}
Let $X$ be a quasi-compact algebraic stack with quasi-finite and
quasi-separated diagonal. Then $X\to \spec \Z$ has an approximation.
In particular, $X$ is pseudo-noetherian.
\end{corollary}
\begin{proof}
By Theorem~\ref{T:qfdiag_pres}, there is an \'etale surjective morphism $W\to
X$ of finite presentation with separated diagonal (a Nisnevich cover) and a finite faithfully flat
morphism $V\to W$ of finite presentation where $V$ is an affine scheme.  We
conclude that $W$ has an approximation by~\cite[Prop.~2.12 (ii)]{rydh-2009} and
that $X$ has an approximation by
Proposition~\ref{P:non-rep-et-descent-of-approx}.
\end{proof}
We can also establish the following improvement of \cite[Thm.~A]{perfect_complexes_stacks} in equicharacteristic~$0$, where it was proved for stacks with quasi-finite and separated diagonal.
\begin{theorem}\label{T:compact-gen}
    Let $X$ be a quasi-compact and quasi-separated Deligne--Mumford stack of \emph{equicharacteristic $0$}. Then the unbounded derived category $\DQCOH(X)$, of
  $\Orb_X$-modules with quasi-coherent cohomology, is
  compactly generated by a single perfect complex. Moreover, for every quasi-compact open subset $U \subseteq X$, there exists a compact perfect complex with support exactly $X\setminus U$. 
\end{theorem}
\begin{proof}
  We apply Theorem \ref{T:induction_qf_diag}: let
  $\mathbf{D} \subseteq
  \mathbf{E}=\STACK_{\mathrm{\sepdiag,fp,\acute{e}t}/X}$ be the
  full subcategory consisting of those morphisms of Deligne--Mumford
  stacks $(W \to X)$, where for every quasi-compact open immersion
  $V \subseteq W$ we have that $V$ satisfies the conclusion of the
  Theorem. This makes condition (I1) a triviality. Condition (I2)
  follows immediately from
  \cite[Thm.~A]{perfect_complexes_stacks}. For Condition (I3) we use
  the theory developed in \cite[\S\S 5-6]{perfect_complexes_stacks},
  with the following minor changes. In
  \cite[Ex.~5.2]{perfect_complexes_stacks}, the working example
  throughout those sections, they take $\mathcal{D}$ to consist of
  representable and finitely presented morphisms to $X$; we will take
  $\mathcal{D} = \mathbf{E}$. The main difference is that
  $\mathcal{D}$ is now a $2$-category, but the results go through
  without change. Since all morphisms of Deligne--Mumford stacks in
  equicharacteristic $0$ are concentrated (combine
  \cite[Lem.~2.5(2)]{perfect_complexes_stacks} with
  \cite[Thm.~C]{hallj_dary_alg_groups_classifying}), the resulting
  $(\mathcal{L},\mathcal{D})$-presheaf of triangulated categories is
  admissible in the sense of
  \cite[Defn.~6.1]{perfect_complexes_stacks}; also see
  \cite[Ex.~6.2]{perfect_complexes_stacks} for further details and
  notations. Condition (I3) now follows from
  \cite[Prop.~6.8]{perfect_complexes_stacks}.
\end{proof}
\begin{corollary}
  If $X$ is a noetherian Deligne--Mumford stack of equicharacteristic
  $0$, then there is an equivalence of categories:
  \[
    \DCAT(\QCOH(X)) \to \DQCOH(X). 
  \]
\end{corollary}
\begin{proof}
  Combine Theorem \ref{T:compact-gen} with \cite[Thm.~1.2]{hallj_neeman_dary_no_compacts}.
\end{proof}
\begin{remark}\label{R:application_LUNA}
  If $p\colon W \to X$ is a morphism of algebraic stacks and $W$ has
  separated diagonal, then $p$ has separated diagonal. This means that
  the \'etale presentations appearing in
  \cite{AHR_lunafield,etale_local_stacks} always have separated
  diagonal.
\end{remark}
\bibliography{references}
\bibliographystyle{bibstyle}
\end{document}